\documentclass{aptpub}

\authornames{S\o ren Asmussen and Sergey Foss} 
\shorttitle{Regular Variation in a Fixed-Point Problem} 

\numberwithin{equation}{section}
\usepackage{mathtools,xspace,booktabs,tikz,url}

\newcommand{\indi}{\mathbb{I}}
\newcommand{\RL}{\mathbb{R}} 
\newcommand{\nat}{{\mathbb N}} 

\newcommand{\dd}{\mathrm{d}} 
\newcommand{\oh}{{\mathrm{o}}} 
\newcommand{\Oh}{{\mathrm{O}}} 

\newcommand{\Exp}{\mathbb{E}}

\newcommand{\Prob}{\mathbb{P}}

\newcommand\convdistr{
  \xrightarrow{\:\scriptscriptstyle\smash{\mathcal{D}}\:}
}
\newcommand{\eqdistr}{\stackrel{{\footnotesize \cal D}}{=}}

\newcommand{\Fb}{{\overline F}}
\newcommand{\zb}{{\overline z}}
\newcommand{\nb}{{\overline n}}
\newcommand{\qb}{{\overline q}}
\newcommand{\rb}{{\overline r}}
\newcommand{\BB}{{\cal B}}

\newcommand{\bfa}{\mbox{\boldmath$a$}}
\newcommand{\bfb}{\mbox{\boldmath$ b$}}
\newcommand{\bfI}{\mbox{\boldmath$ I$}}
\newcommand{\bfN}{\mbox{\boldmath$ N$}}
\newcommand{\bffN}{\mbox{\scriptsize\boldmath$ N$}}

\newcommand{\bfS}{\mbox{\boldmath$ S$}}
\newcommand{\bfT}{\mbox{\boldmath$ T$}}
\newcommand{\bfV}{\mbox{\boldmath$ V$}}
\newcommand{\bffV}{\mbox{\scriptsize\boldmath$ V$}}
\newcommand{\bfv}{\mbox{\boldmath$ v$}}
\newcommand{\bfX}{\mbox{\boldmath$X$}}
\newcommand{\bfM}{\mbox{\boldmath$ M$}}
\newcommand{\bfz}{\mbox{\boldmath$ z$}}
\newcommand{\bfZ}{\mbox{\boldmath$ Z$}}
\newcommand{\bftheta}{\mbox{\boldmath$ \theta$}}
\newcommand{\bfTheta}{\mbox{\boldmath$ \Theta$}}
\newcommand{\compl}{{\mbox{\normalsize c}}}
\newcommand{\MRVF}{{\mbox{{\rm MRV}$(F)$}}}

\newcommand{\rev}[1]{#1}

\newcommand{\tree}{{\mathcal G}}

\newcommand{\gb}{node {$\bullet$}}  \newcommand{\gbb}{node {{\footnotesize$\blacktriangle$}}}
\newcommand{\treefig}{
\begin{figure}[htb]
   \centering  
\begin{tikzpicture}[xscale=1,yscale=0.5] 
\draw[thin] (0,4) -- (1,4) (1,1) -- (1,7) (1,1) -- (2,1) (1,2) -- (2,2) (1,4) -- (3,4) (1,5) -- (2,5) (1,7) -- (2,7);
\draw[thin] (2,0.5) -- (2,1.5) (2,6.5) -- (2,7.5) (2,1.5) -- (4,1.5) (2,6.5) -- (3,6.5) (2,7.5) -- (4,7.5) (3,3.5) -- (3,4.5);
\draw[green] (0,4) \gb (1,1) \gb (1,2) \gb (1,3) \gb (1,5) \gb (1,6) \gb;
\draw[red] (1,4) \gb (1,1) (1,7) \gb (1,1); 
\draw[green] (2,0.5) \gb (2,1.5) \gb (2,4) \gbb (2,5) \gb (2,7) \gbb;
\draw[red] (2,2) \gb (2,6.5) \gb (2,7.5) \gb;
\draw[blue] (2,2) \gb;
\draw[green] (3,3.5) \gb (3,7.5) \gbb (4,1.5) \gb;
\draw[red] (3,6.5) \gb;
\draw[blue] (3,1.5) \gb  (3,4.5) \gb (4,7.5) \gb;
\end{tikzpicture}
\caption{Reducing from 2  types to 1}
\label{treefig}
\end{figure}
}

\begin{document}

\title{Regular Variation in a \rev{Fixed-Point} Problem for Single- and Multiclass Branching Processes and Queues } 

\authorone[Aarhus University]{S\O ren Asmussen} 
\addressone{Department of Mathematics, Aarhus University, Ny Munkegade, 8000 Aarhus C, Denmark} 
\authortwo[Heriot-Watt  University and\\ 
. \hspace{3.1cm} Novosibirsk State University]{Sergey Foss}
\addresstwo{School of Mathematical and Computer Sciences, Heriot-Watt University, EH14 4AS, Edinburgh, United Kingdom. Research  supported by RSF grant No. 17-11-01173.}

\begin{abstract}
Tail asymptotics of the solution $R$ to a \rev{fixed-point} problem of type $R \ \eqdistr\ Q + \sum_1^N R_m$
is derived under heavy-tailed conditions allowing both dependence between $Q$ and $N$ and the 
tails to be of the same order of magnitude. Similar results are derived for a $K$-class version with applications
 to multitype branching processes and busy periods in multiclass queues.\end{abstract}



\section{Introduction}\label{S:Intr}
\setcounter{equation}{0}

This paper is concerned with the tail asymptotics of the solution $R$ 
to the \rev{fixed-point} problem
\begin{equation}\label{eq1}
 R \ \eqdistr\ Q + \sum_{m=1}^N R_m
\end{equation}
under suitable regular variation (RV) conditions
and the similar problem in a multidimensional setting stated below as \eqref{30.6a}. Here 
in \eqref{eq1} $Q,N$ are (possibly dependent) non-negative non-degenerate r.v.'s where $N$ is integer-valued, $R_1,R_2,\ldots$
are i.i.d.\ and distributed as $R$, and $\nb=\Exp N<1$ (similar notation for expected values is used in the following).

A classical example is $R$ being the M/G/1 busy period, cf.\ \cite{deMeyer}, \cite{Bert}, where $Q$ is the service
time of the first customer in the busy period and $N$ the number of arrivals during his service. Here 
$Q$ and $N$ are indeed heavily dependent, with tails of the same order of magnitude when $Q$ has a regularly varying (RV) distribution; more precisely,
$N$ is Poisson$(\lambda q)$ given $Q=q$. Another example is the total progeny of a subcritical branching
process, where $Q\equiv 1$ and $N$ is the number of children of the ancestor, More generally, $R$ could be
the total life span of the individuals in a Crump-Mode-Jagers process (\cite{Jagers}), corresponding to 
$Q$ being the lifetime of the ancestor and $N$  the number of her children. Related examples are weighted
branching processes, see~\cite{Mariana} for references. Note that connections between branching processes and RV
have a long history, going back at least to \cite{Seneta69}, \cite{Seneta74}.

Recall some definitions of classes of heavy-tailed distributions.
A distribution $F$ on the real line is {\it long-tailed}, $F\in {\cal L}$, 
if, for some $y>0$
\begin{equation}\label{long}
\frac{\overline{F}(x+y)}{\overline{F}(x)} \rightarrow 1\quad\text{as}\ x\to\infty ; 
\end{equation}
$F$ is {\it regularly varying}, $F\in {\cal RV}$, 
if, for some 
$\beta >0$, 
$
\overline{F}(x) = x^{-\beta} L(x),
$
where $L(x)$ is a {\it slowly varying} (at infinity) function;\\
$F$ is {\it intermediate regularly varying}, $F\in {\cal IRV}$,
if
\begin{equation}\label{IRV1}
\lim_{\alpha \uparrow 1} \limsup_{x\to\infty}
\frac{\overline{F}(\alpha x)}{
\overline{F}(x)} = 1.
\end{equation}
It is known that ${\cal L} \supset {\cal IRV} \supset {\cal RV}$
and if $F$ has a finite mean, then ${\cal L} \supset {\cal S}^*
\supset {\cal IRV}$
where ${\cal S}^*$ is the class of so-called {\it strong subexponential
distributions}, see e.g. \cite{EKM1997} or \cite{FKZ2013} for further
definitions and properties of heavy-tailed distributions.

Tail asymptotics of quantities related to $R$ have earlier been studied in
\cite{Mariana}, \cite{Litvak} under RV conditions (see also~\cite{BDM}). 
 Our main result is the following: 
\begin{theorem}\label{Th:3.8a}
Assume $\nb<1$ and $ \qb<\infty$. Then:\\[1mm]
{\rm (i)}
There is only one \rev{non-negative} solution $R$ to equation \eqref{eq1} with finite mean. For this solution,
$\rb=\qb/(1-\nb)$.\\[1mm]
{\rm (ii)} If further\\
{\bf (C)}
the distribution of $Q+cN$ is intermediate regularly varying for
all $c>0$ in the interval $(\rb -\epsilon, \rb+\epsilon )$ where $\rb$ is as in {\rm (i)} and
$\epsilon >0$ is any small number,\\
then 
\begin{equation}\label{eq10}
 \Prob   (R>x) \sim \frac{1}{1-\nb} \Prob(Q+\rb 
N >x)\quad \text{ as }x\to\infty\,.
\end{equation}
{\rm (iii)} In particular, condition {\bf (C)} holds in the following three cases:\\
{\rm (a)} $(Q,N)$ has a 2-dimensional regularly varying distribution;\\
{\rm (b)} $Q$ has an intermediate regularly varying distribution and $\Prob   (N>x) = \oh(\Prob   (Q>x))$;\\
{\rm (c)} $N$ has an intermediate regularly varying distribution and $\Prob   (Q>x) = \oh(\Prob   (N>x))$.
\end{theorem}

Part (i) is well known from several sources and not deep (see 
the proof of the more general
Proposition~\ref{Prop:6.8a} below and the references at the end of
the section for more general versions). 
Part (ii) generalizes and unifies  results of \cite{Mariana}, \cite{Litvak} in several ways. 
Motivated from the Google page rank
algorithm, both of these papers consider a more general recursion 
\begin{equation}\label{eq1A}
 R \ \eqdistr\ Q + \sum_{m=1}^N A_mR_m\,.
\end{equation} 
However, \cite{Mariana} does not allow dependence and/or the tails of $Q$ and $N$ to be equally heavy.
These features are incorporated in \cite{Litvak}, but on the other hand that paper require strong conditions
on the $A_i$ which do not allow to take $A_i\equiv 1$ when dealing with sharp asymptotics. To remove all of these restrictions is essential for
the applications to queues and branching processes we have in mind.  
Also, our proofs are considerably simpler and shorter
than those of \cite{Mariana}, \cite{Litvak}. The key tool is a general result of \cite{FossZ}
giving the tail asymptotics of the maximum of a random walk up to a (generalised) stopping time.

\begin{remark}
\rev{In Theorem \ref{Th:3.8a}, we considered the case $A_i\equiv 1$ only. However, our approach may work in the more
general setting of~\eqref{eq1A} with i.i.d. positive $\{A_m\}$ that do not depend on $Q,N$ and $\{R_m\}$. For example, if we assume, in addition to
$\nb <1$, 
that $\Prob (0<A_1\le 1) =1$,
then the exact tail asymptotics for $\Prob (R>x)$ may be easily found using the upper bound
\eqref{eq10} and the principle of a single big jump. However, the formula
for the tail asymptotics in this case is much more complicated that \eqref{eq10}.} 
\end{remark}

The multivariate version involves  a \rev{set} $\bigl(R(1),\ldots,(R(K)\bigr)$ of r.v.'s satisfying
\begin{equation}\label{30.6a}R(i) \ \eqdistr\  Q(i)+ \sum_{k=1}^K\sum_{m=1}^{N^{(k)}(i)}R_{m}(k)\end{equation}
In the branching setting, this relates to $K$-type   processes by thinking of $N^{(k)}(i)$ as the number of type $k$
children of a type $i$ ancestor.
One example is the total progeny where $Q(i)\equiv 1$, others relate as above to the total life span and
weighted branching processes.
A queueing example is the busy periods $R(i)$ in the multiclass queue in \cite{AEH}, with $i$ the class of the first
customer in the busy period and $Q(i)$   the service time of a class $i$ customer; the model states that during
service of a class $i$ customer, class $k$ customers arrive at rate $\lambda_{ik}$.  One
should note for this example  \cite{AEH} \rev{gives} only lower asymptotic bounds,
whereas we here provide sharp asymptotics.

The treatment of \eqref{30.6a} is considerably more involved than for \eqref{eq1}, and we defer the details
of assumptions and results to Section~\ref{S:MRV}. We remark here only that the concept of multivariate
regular variation (MRV) will play a key role; that
the analogue of the crucial assumption $\nb<1$ above is subcriticality, $\rho=$spr$(\bfM)<1$ where spr means spectral radius and
$\bfM$ is the offspring mean matrix with elements
$m_{ik}=\Exp N^{(k)}(i)$; and that the argument will involve a recursive procedure from \cite{Foss1, Foss2}, 
reducing $K$ to $K-1$ so that in the end we are back to  the case $K=1$ of \eqref{eq1} and Theorem~\ref{Th:3.8a}.
\rev{\subsubsection*{Bibliographical remarks}
 An $R$, or its distribution, satisfying \eqref{eq1A} is often called a fixed point of the smoothing transform (going back to~\cite{Liggett}). There is an extensive literature
 on this topic, but rather than on tail asymptotics, the emphasis is most often on existence and uniqueness questions
(these are easy in our context with all r.v.'s non-negative with finite mean and we give
short self-contained proofs). Also the assumption $A_i\neq 1$ is crucial for most of this
literature. See further \cite{Aldous}, \cite{Gerold1}, \cite{Gerold2}
and references there.}

\rev{It should be noted that the term ``multivariate smoothing transform" 
(e.g.~\cite{Mentemeier})  means to a recursion
of vectors, that is, a version of \eqref{eq1} with $R,Q\in \RL^K$. This is different  from our set-up because in~\eqref{30.6a} we are only interested in the one-dimensional
distributions of the $R(i)$. In fact, for our applications there is no interpretation of a vector  with $i$th marginal having the distribution of $R(i)$.
 }

In \cite{Vatutin}, tail asymptotics for the
total progeny of a multitype branching process is studied by 
different techniques€‹ in the critical case $\rho =1$.

\section{One-dimensional case: equation \eqref{eq1}}\label{S:1D}
\setcounter{equation}{0}

The heuristics behind \eqref{eq10} is the principle of a single large jump: for $R$ to exceed $x$, either one or both elements of $(Q,N)$
must be large, or the independent
event occurs that $R_m>x$ for some   $m\le N$, in which case $N$ is small or moderate.
If $N$ is large, $\sum_1^NR_m$ is approximately $\rb N$, so roughly the probability of the first possibility is 
$\Prob(Q+\rb N>x)$. On the other hand,   results for compound heavy-tailed sums suggest that the  approximate probability of the second
possibility is $\nb\Prob(R>x)$. We thus arrive at
\[\Prob(R>x) \approx\ \Prob(Q+\rb N>x)\,+\,\nb\Prob(R>x)\]
and \eqref{eq10}.

%


In the proof of Theorem~\ref{Th:3.8a},  let $(Q_1,N_1),(Q_2,N_2),\ldots$ be an i.i.d.\ sequence of pairs distributed as the (possibly dependent)
pair $(Q,N)$ in \eqref{eq1}. Then $S_n=\sum_{i=1}^n \xi_i$, $i=0,1,\ldots$ where $\xi_i=N_i-1$
is a random walk. Clearly,
$\Exp    \xi_i <0$. Let
\begin{equation}\label{eq2}
\tau = \min \{n\ge 1 \ : \ S_n<0\}= \min \{n\ge 1 \ : \ S_n=-1\}\,.
\end{equation}
Note that by Wald's identity $\Exp S_\tau=\Exp\tau\cdot\Exp (N-1)$ and $S_\tau=-1$ we have
\begin{equation}\label{3.8b}
\Exp\tau = \frac{1}{1-\Exp N}
\end{equation}
Now either $N_1=0$, in which case $\tau=1$, or $N_1>0$ so that $S_1=N_1-1$ and to proceed to level -1,
the random walk must go down one level $N_1$ times. This shows that  (in obvious notation)
\begin{equation}\label{3.8d}
 \tau \ \eqdistr\ 1 + \sum_{i=1}^N \tau_i
\end{equation}
That is, $\tau$ is a solution to \eqref{eq1} with
$Q\equiv 1$. On the other hand, the total progeny in a Galton-Watson process with the number of offsprings
of an individual distributed as $N$ obviously also satisfies \eqref{3.8d}, and hence by uniqueness
must have the same distribution as $\tau$. This result first occurs as equation (4) in \cite{Dwass}, but note that
an alternative representation (1) in that paper appears to have been the one receiving the most attention
in the literature.

Now define $\varphi_i=k_0+k_1Q_i$,
\begin{equation}\label{eq3}
V=\sum_{i=1}^{\tau} \varphi_i
\end{equation}
Here the $k_0,k_1$ are non-negative constants, $k_0+k_1>0$.
In particular, if $k_0=1,k_1=0$, then $V=\tau$, and 
further
\begin{equation}\label{3.8c}
k_0=0, k_1=1\qquad\Rightarrow\qquad V\eqdistr R.
\end{equation}
\rev{Indeed, arguing as before, we conclude that equation $V\eqdistr \varphi +\sum_1^N V_i$ has only one integrable positive solution, and, clearly,
$$
V \eqdistr \varphi + \sum_1^N V_i \eqdistr \varphi + \sum_1^N \varphi_i + \sum_1^N \sum_1^{N_i} \varphi_{i,j} + \sum_1^N \sum_1^{N_i} \sum_1^{N_{i,j}} \varphi_{i,j,k} +
\ldots \eqdistr \sum_1^{\tau} \varphi_i
$$
where, like before, $(\varphi, N)$, $(\varphi_i,N_i)$, $(\varphi_{i,j},N_{i,j})$, etc. are
i.i.d. vectors. In particular, $V$ becomes $R$ when replacing $\varphi$ by $Q$.}

\begin{proof}[Proof of Theorem~\ref{Th:3.8a}]
It remains to find the asymptotics of $\Prob (V>x)$ as
$x\to\infty$. Throughout the proof, we assume $k_1>0$.

Let $r_0$ be the solution to the 
equation
$$
\Exp    \varphi_1+r_0 \Exp    \xi_1=0.
$$
Note that in the particular case where $k_0=0$ and $k_1=1$,
\begin{equation}\label{eq5}
r_0= \frac{\Exp    Q}{1-\Exp    N} = \overline{r}.
\end{equation} 
Choose $r>r_0$ as close to $r_0$ as needed  and let
$$
\psi_i=\varphi_i+ r\xi_i.
$$
We will find upper and lower bounds for the asymptotics
of
$\Prob(V>x)$ and show that they are asymptotically equivalent.

Since $k_1>0$ and $Q+Nr/k_1$
has an IRV distribution, the distribution of $k_1Q+rN$ is IRV, too. 
\\[2mm]
{\bf Upper bound.} The key is to apply  the main result of~\cite{FossZ} 
to obtain the
following upper bound.
\begin{align*}\MoveEqLeft
\Prob  (V>x) \ = \
\Prob  \Bigl(\sum_{i=1}^{\tau} \varphi_i > x\Bigr) \ = \
\Prob \Bigl(\sum_{i=1}^{\tau} \psi_i > x + r 
{S}_{\tau}\Bigr)\\
&=\
\Prob \Bigl(\sum_{i=1}^{\tau} \psi_i > x-r\Bigr)\ \le\
\Prob  \Bigl(\max_{1\le k \le \tau} 
\sum_{i=1}^{k} \psi_i > x-r\Bigr)\\
&\sim \
\Exp \tau 
\Prob  (\psi_1 > x-r) \ \sim \
\Exp \tau 
\Prob  (\psi_1 > x-r+k_0) \ = \
\Exp \tau 
\Prob  (k_1Q+rN > x)
\end{align*}
Here the first equivalence follows from~\cite{FossZ}, noting that   the distribution of $\psi_1$
belongs to the class $S^*$ and that~\cite{FossZ} only requires $\varphi_1, \varphi_2,\ldots$ to be i.i.d.\ 
w.r.t.\ some filtration w.r.t.\ which $\tau$ is a stopping time. For the second, we used
the long-tail property \eqref{long} of the distribution of $\psi_1$.

Let $F$ be the distribution function of $k_1Q+r_0N$.
 Then, as $x\to\infty$,  
 $$
 \overline{F}(x) \le \Prob (k_1Q+rN>x) \le 
 \Prob (rk_1Q/r_0+rN>x) \le \overline{F}(\alpha x) \le
 (1+o(1)) c(\alpha )\overline{F}(x)
 $$ 
 where $\alpha = r_0/r<1$ and $c(\alpha ) =\limsup_{y\to\infty}  \overline{F}(\alpha y)/\overline{F}(y)$.
 
 Now we assume the IRV condition to hold, let $r\downarrow r_0$ and apply
 \eqref{IRV1} to obtain the  upper bound
\begin{equation}\label{eq4}
 \Prob   (R>x) \le (1+o(1))
\Exp    \tau \Prob   (k_1Q+r_0N>x)
\end{equation}
In particular, if $k_0=0$ and $k_1=1$, then
$r_0=\overline{r}$ is as in~\eqref{eq5}. \\[2mm]
{\bf Lower bound.}
 Here we put $\psi_n = \varphi_n + r\xi_n$ where
$r$ is any positive number strictly smaller than $r_0$. 
Then the $\psi_n$ are i.i.d.\ random variables with common mean $\Exp \psi_1 >0$. 

 We have, for any fixed  $C>0$, $L>0$,   $n=1,2,\ldots$  and $x\ge 0$ that  
 \begin{equation}\label{double}
 \Prob (V>x) \ \ge \
\Prob\Bigl(\sum_{i=1}^{\tau } \psi_i >x\Bigr)\   \ge\  \sum_{i=1}^n \Prob (D_i \cap A_i)
 \end{equation}
 where
\[
 D_i \ =\  \Bigl\{ \sum_{j=1}^{i-1} |\psi_j|\le C, \tau \ge i,
\psi_i>x+C+L \Bigr\}
\quad\text{and}\quad
A_i \ =\  \bigcap_{\ell\ge 1}\Bigl\{\sum_{j=1}^\ell \psi_{i+j}\ge -L \Bigr\}. 
\]
Indeed, the first inequality in \eqref{double} holds since $S_{\tau}$ is non-positive. Next, the events $D_i$
are disjoint and, given  $D_i$, we have $\sum_1^i \psi_j >x+L$. Then, given  $D_i\cap A_i$,
we have $\sum_1^k \psi_j \ge x$ for all $k\ge i$ and, in particular, $\sum_{j=1}^{\tau}\psi_j >x$. Thus, \eqref{double} holds.

The events $\{A_i\}$ form a stationary sequence. Due to the SLLN, for any $\varepsilon >0$,  
one can choose
$L=L_0$ so large that  $\Prob (A_i) \ge 1-\varepsilon$. 

For this $\varepsilon$, choose $n_0$ and $C_0$ such that 
$$
\sum_{i=1}^{n_0} \Prob \Bigl(\sum_{j=1}^{i-1} |\psi_j|\le C_0, \tau \ge i\Bigr)\ \ge\ (1-\varepsilon ) \Exp \tau.
$$
Since the random variables $(\{\psi_j\}_{j<i},{\mathbf I}(\tau\le i))$ are independent of $\{\psi_j\}_{j\ge i}$, we obtain further that, for any
$\varepsilon \in (0,1)$ and for any $n\ge n_0$, $C\ge C_0$ and $L\ge L_0$,  
\begin{eqnarray*}
\Prob (V>x) &\ge &
\sum_{i=1}^n \Prob \Bigl(\sum_{j=1}^{i-1} |\psi_j|\le C, \tau \ge i\Bigr)
\Prob (\psi_i>x+C+L) \Prob (A_i) \\
&\ge &
(1-\varepsilon )^2 \Prob (\psi_1 >x+C+L) \sum_{i=1}^n \Prob (\tau \ge i)\\
&\sim &
(1-\varepsilon )^2 \Prob (\psi_1 >x) \sum_{i=1}^n \Prob (\tau \ge i),
\end{eqnarray*}
as $x\to\infty$. Here the final equivalence   follows from the long-tailedness
of   $\psi_1$. Letting first $n$ \rev{go} to infinity and then $\varepsilon$ to zero, we get $\liminf_{x\to\infty} \Prob (V>x)/\Exp \tau \Prob (\psi_1>x) \ge 1.$ Then we let $r\uparrow r_0$ and use the IRV property \eqref{IRV1}.  In the particular case
$k_0=0,k_1=1$ we obtain an asymptotic lower bound that is equivalent
to the upper bound derived above  
\end{proof} 

\begin{remark}
A slightly more intuitive approach to the lower bound is to bound $\Prob(R>x)$ below by the sum of the contributions from the
disjoint events $B_1,B_2,B_3$ where \[
B_1=B\cap\{\rb N\ >\epsilon x\},\quad B_2=B\cap\{A< \rb N\le\epsilon x\},\quad
B_3= \{\rb N\le A\}\]
with $B=\{Q+\rb N>(1+\epsilon) x\}$.
Here for large $x,A$ and small $\epsilon$,
\begin{align*}\Prob(R>x;\,B_1)\ &\sim\ \Prob(Q+\rb N>x,\,\rb N>\epsilon x)\\
\Prob(R>x;\,B_2)\ &\ge \Prob(Q>x,\rb N\le\epsilon x)\ \sim \ \Prob(Q+\rb N>x,\,\rb N\le\epsilon x)\\
\Prob(R>x;B_3)\ &\ge\ \sum_{n=0}^{A/\rb}\Prob(R_1+\cdots+R_n>x)\Prob(N=n)\\ & \ \ge\
\sum_{n=0}^{A/\rb}\Prob\bigl(\max(R_1,\ldots,R_n)>x\bigr)\Prob(N=n)
\\ & \ \sim\ \sum_{n=0}^{A/\rb}n\Prob(R>x)\Prob(N=n)\ \sim\ \Exp[N\wedge A/\rb]\Prob(R>x) \ \sim\ \nb \Prob(R>x) 
\end{align*}
We omit the full details since they are close to arguments given in Section~\ref{S:FayeMRV} for the
multivariate case.\end{remark}

\section{Multivariate version}\label{S:MRV}
\setcounter{equation}{0}

\rev{The assumptions for  \eqref{30.6a} are that all $R_{m}(k)$ are independent of
the vector
\begin{equation}\label{30.6ex}\bfV(i)\ =\ \bigl(Q(i),N^{(1)}(i),\ldots,N^{(K)}(i)\bigr)\,,
\end{equation}
that they are mutually independent and that $R_{m}(k)\eqdistr R(k)$.
Recall that we are only interested in the one-dimensional
distributions of the $R(i)$. Accordingly, for a solution to \eqref{30.6a} 
we only require the validity for each fixed $i$. }

Recall that the offspring mean matrix is denoted $\bfM$ where $m_{ik}=\Exp N^{(k)}(i)$,
and that $\rho=$spr($\bfM)$; $\rho$ is the Perron-Frobenius root if $\bfM$ is irreducible which it is
not necessary to assume. \rev{No restrictions on the dependence structure of the vectors in \eqref{30.6ex} }need to be imposed for the following result to hold (but later we need MRV!):

\begin{proposition}\label{Prop:6.8a} Assume $\rho<1$. Then:\\[1mm] 
{\rm (i)} the \rev{fixed-point} problem \eqref{30.6a} has a unique non-negative solution 
with $\rb_i=\Exp R(i)<\infty$ for all $i$;\\[1mm]  {\rm (ii)} the $\rb_i=\Exp R(i)<\infty$ are given as the unique solution to the set
\begin{equation}\label{6.8dd} \rb_i\ =\ \qb_i+\sum_{k=1}^K m_{ik}\rb_k\,,\qquad i=1,\ldots,K,
\end{equation}
of linear equations.
\end{proposition} 
\begin{proof} (i) Assume first $Q(i)\equiv 1$, $i=1,\ldots,K$. \rev{The existence of a
solution to \eqref{30.6a} is then clear since we may take $R(i)$ as the total progeny
of a type $i$  ancestor in a  $K$-type Galton-Watson process where the vector of children of a type $j$ individual is distributed as $\bigl(N^{(1)}(j),\ldots, N^{(K)}(j)\bigr)$.
For uniqueness, let $\bigl(R(1),\ldots, R(K)\bigr)$ be any solution and}
consider the $K$-type Galton-Watson trees $\tree(i)$, $ i=1,\ldots,K$,  where $\tree(i)$ corresponds to an
ancestor of type $i$.
If we define $R^{(0)}(i)=1$,
\[R^{(n)}(i) \ \eqdistr\  1+ \sum_{k=1}^K\sum_{m=1}^{N^{(k)}(i)}R^{(n-1)}_{m}(k)\,,\]
\rev{with similar conventions as for \eqref{30.6a},}
then $R^{(n)}(i)$  is the  total progeny of a type $i$  ancestor under the restriction
that the depth of the tree is at most $n$. Induction easily gives that 
\rev{$R^{(n)}(i)\preceq_{{\rm st}} R(i)$ ($\preceq_{{\rm st}} \ =\ $ stochastic order) for 
each $i$. Since also $R^{(n)}(i) \preceq  R^{(n+1)}(i)$,   limits
$R^{(\infty)}(i)$ exist,   $R^{(\infty)}(i)$ must simply be the unrestricted vector of total progeny of different types,
and $ R^{(\infty)}(i)\preceq_{{\rm st}}   R(i)$. 
Assuming the $R(i)$} to have finite mean, \eqref{6.8dd} clearly holds with
$\qb_i=1$, and so the $\Delta_i=\rb_i-\Exp R^{(\infty)}(i)$ satisfy $\Delta_i=\sum_1^K m_{ik}\Delta_k$. But
$\rho<1$ implies that  $\bfI-\bfM$ is invertible so the only solution is $\Delta_i=0$  which in view of
\rev{$R^{(\infty)}(i)\preceq_{{\rm st}}  R(i)$ implies $R^{(\infty)}(i) \eqdistr  R(i)$} and the stated uniqueness when $Q(i)\equiv 1$.

For more general $Q(i)$, we  equip each individual of type $j$ in $\tree(i)$ with a weight distributed as $Q(j)$,
such that the dependence between her $Q(j)$ and her offspring vector has the given structure. The argument
is then  a straightforward generalization and application of what was done above for $Q(i)\equiv 1$.

(ii) Just take expectations in  \eqref{30.6a} and note as before that  $\bfI-\bfM$ is invertible.
\end{proof}

For  tail asymptotics,  we need an MRV assumption. The definition of MRV exists in some equivalent variants,
cf.\ \cite{Sid87}, \cite{Meerschaert}, \cite{Basrak}, \cite{Sid07}, but we shall use the one in polar $L_1$-coordinates
adapted to deal with several random vectors at a time as in \eqref{30.6ex}. Fix here and in the following
a reference RV tail $\Fb(x)=L(x)/x^\alpha$  on $(0,\infty)$,
for $\bfv=(v_1,\ldots, v_p)$ define $\|\bfv\|=\|\bfv\|_1=$ $|v_1|+\cdots+|v_p|$ and let $\BB=\BB_p=$ $\{\bfv:\, \|\bfv\|=1\}$.
We then say that a random vector $\bfV=(V_1,\ldots, V_p)$ 
satisfies \MRVF\ or has property \MRVF\  if 
$\Prob\bigl(\|\bfV\|>x\bigr)$ $\sim b\Fb(x)$ where either (1) $b=0$ or (2) $b>0$ 
 and the angular part $\bfTheta_{\bffV}=\bfV/\|\bfV\|$ satisfies
\[\Prob\bigl(\bfTheta_{\bffV}\in\cdot\,\big|\,\|\bfV\|>x\bigr)\ \convdistr \ \mu \text{ as }x\to\infty\]
for some measure $\mu$   on $\BB$ (the angular measure). 
Our basic condition is then that for the given reference RV tail $\Fb(x)$ we have that\\[1mm]
(MRV) For any $i=1,\ldots,K$ the vector $\bfV(i)$ in \eqref{30.6ex} satisfies \MRVF, where
$b=b(i)>0$ for at least one $i$.\\[1mm]
Note that $F$ is the same for all $i$ but the angular measures $ \mu_i$ not necessarily so.
We also assume that the mean $\zb$ of $F$ is finite, which will ensure that all expected values coming up in the following
are finite.

\rev{Assumption  \MRVF\ implies RV of linear combinations, in particular 
marginals. More precisely (see the Appendix),}
\begin{equation}\label{30.6ax}\Prob\bigl(a_0Q(i)+a_1N^{(1)}(i)+\cdots+a_KN^{(K)}(i) >x\bigr)\ \sim\ c_i(a_0,\ldots,a_K)\Fb(x)
\end{equation} where $\displaystyle
c_i(a_0,\ldots,a_K)\ =\ b(i)\int_{{\mathcal B}}(a_0\theta_0+\cdots+a_K\theta_K)^\alpha\mu_i(\dd\theta_0,\ldots,\dd\theta_K)\,.$
\begin{theorem}\label{Th:6.8a} Assume that $\rho<1$, $\zb<\infty$ and that \emph{(MRV)} holds.
Then there are constants $d_1,\ldots,d_K$ such that
\begin{equation}\label{30.6d}
\Prob(R(i)>x)\,\sim\,d_i\Fb(x)\ \ \text{as}\ x\to\infty.
\end{equation}
Here the $d_i$  are given as the unique solution to the set
\begin{equation}\label{6.8d} d_i\ =\ c_i(1,\rb_1,\ldots,\rb_K)+\sum_{k=1}^K m_{ik}d_k\,,\qquad i=1,\ldots,K,
\end{equation}
of linear equation\rev{s} where the $\rb_i$ are as in Proposition~\ref{Prop:6.8a} and the $c_i$ as in
\eqref{30.6ax}.
\end{theorem}
\noindent The proof follows in Sections~\ref{S:Outline}--\ref{S:ProofCompl}.

\section{Outline of proof}\label{S:Outline}
\setcounter{equation}{0}

When $K>1$, we did not manage to find a random walk argument extending Section~\ref{S:1D}. 
Instead, we shall use a recursive procedure, going back to \cite{Foss1, Foss2} in a queueing setting, for eventually
being able to infer \eqref{30.6d}. The identification  \eqref{6.8d} of the $d_i$ then follows immediately
from the following result to be proved in Section~\ref{S:FayeMRV} (the case $p=1$ is Lemma 4.7 of \cite{Fay}):

\begin{proposition}\label{Prop:2.8b} Let $\bfN=(N_1,\ldots,N_p)$ be MRV with 
$\Prob\bigl(\|\bfN\|>x\bigr)\sim c_{\bffN}\Fb(x)$ and let the r.v.'s $Z_m^{(i)}$, $i=1,\ldots,p$, $m=1,2,\ldots$, be
independent with distribution $F_j$ for $Z_i^{(j)}$, independent of $\bfN$ 
and having finite mean $\overline z_j=\Exp Z_m^{(j)}$. Define
$S_{m}^{(j)}=Z_1^{(j)}+\cdots+Z_{m}^{(j)}$. If $\Fb_j(x)\sim c_j\Fb(x)$, then \[
\Prob\bigl(S_{N_1}^{(1)}+\cdots+S_{N_p}^{(p)}>x\bigr)\ \sim\ \Prob(\zb_1N_1+\cdots+\zb_1N_p\,>\,x)+c_0\Fb(x)\]
where $c_0\,=\,c_1\Exp N_1+\cdots+c_p\Exp N_p$ . \end{proposition}

The recursion idea in \cite{Foss1, Foss2} amounts in a queueing context to let
all class $K$ customers be served first. We implement it here in the branching setting.
Consider the multitype Galton-Watson  tree
 $\tree$. 
For an ancestors of type $i<K$ and any of her daughters $m=1,\ldots,  N^{(K)}(i)$
of type $K$, consider the family tree $\tree_m(i)$ formed by $m$ and all her type $K$ descendant in \emph{direct}
line. For \rev{a vertex} $g\in\tree_m(i)$ and $k<K$, let $N^{(k)}_g(K)$ denote the number of type $k$ daughters
of $g$. 

Note that $\tree_m(i)$ is simply a one-type Galton-Watson tree with the number
of daughters distributed as $N^{(K)}(K)$ and starting from a single ancestor.
In particular, the expected size of $\tree_m(i)$ is $1/(1-m_{KK})$. We further have
\begin{align}\label{30.6aa}R(i) \ &\eqdistr\  \widetilde Q(i)+ \sum_{k=1}^{K-1}\sum_{m=1}^{\widetilde N^{(k)}(i)}R_{m}(k)\,,\quad i=1,\ldots,K-1,
\\ \intertext{where} \label{30.6aaa}
\widetilde Q(i)\ &=\ Q(i)+\sum_{m=1}^{N^{(K)}(i)}\sum_{g\in\tree_m(i)} Q_g(K)\,, \\ \label{30.6aab}
\widetilde N^{(k)}(i) \ &=\  N^{(k)}(i)+\sum_{m=1}^{N^{(K)}(i)}\sum_{g\in\tree_m(i)}N^{(k)}_g(K) \,.
\end{align}
that is, a \rev{fixed-point} problem with one type less.

\begin{example}\label{Ex:5.8a}\rm Let $K=2$ and consider the 2-type family tree in Fig.~\ref{treefig}, 
where type $i=1$ has 
green color, the type $2$ descendants of the ancestor in direct line red, and the remaining type 2 individuals blue.
The green type 1 individuals marked with a triangle are the ones that are counted as extra type 1 children
in the reduced recursion \eqref{30.6aa}.
We have $N^{(2)}(1)=2$ and if $m$ is the upper red individual 
of type 2, then $\tree_m(2)$ has size 4. Further $\sum_{g\in\tree_m(1)}N^{(1)}_g=2$, with $m$ herself and her upper
daughter each contributing with one individual.

\treefig

The offspring mean in the reduced 1-type family tree is $\widetilde m=m_{11}+m_{12}m_{21}/(1-m_{22})$.
Indeed, the first term is the expected number of original type 1 offspring of the ancestor and in the second term,
$m_{12}$ is the expected number of  type 2 offspring of the ancestor, $1/(1-m_{22})$ the size of the direct line
type 2 family tree produced by each of them, and $m_{21}$ the expected number of type 1 offspring 
of each individual in this tree. 

Since the original 2-type tree is finite, the reduced 1-type  tree must necessarily also
be so, so that $\widetilde m\le 1$. A direct verification of this is instructive. First note that  
\[ \widetilde m\le 1\ \iff\  m_{11}-m_{11}m_{22}+m_{12}m_{21}\le 1-m_{22}\ \iff\  \text{tr}(\bfM)-\text{det}(\bfM)\le1\]
But the characteristic polynomial of the
2-type offspring mean matrix $\bfM$  is $\lambda^2-\lambda\,\text{tr}(\bfM)+\text{det}(\bfM)$.
Further the dominant eigenvalue $\rho$ of $\bfM$ satisfies $\rho< 1$ so that
\[ \text{tr}(\bfM)-\text{det}(\bfM)\ \le\ \rho\,\text{tr}(\bfM)-\text{det}(\bfM)\ =\ \rho^2\ <\ 1.\]
\end{example}

\section{Proof of Proposition~\ref{Prop:2.8b}}\label{S:FayeMRV}
\setcounter{equation}{0}

We shall need  the following result of Nagaev et al.\ (see the discussion in \cite{Fay}
around  equation (4.2) there for references):
\begin{lemma}\label{Lemma:11.7aa} Let $Z_1,Z_2,\ldots$ be i.i.d.\ and RV with finite mean $\overline z$ and
define $S_k=Z_1+\cdots+Z_k$. Then for any $\delta>0$
\[ \sup_{y\ge \delta k}\Bigl|\frac{\Prob(S_k>k\zb+y)}{k\Fb(y)}-1\Bigr|\ \to\ 0,\ k\to\infty.\]
\end{lemma}
\begin{corollary}\label{Cor:2.8a} Under the assumptions of Lemma~\ref{Lemma:11.7aa},
it holds for $0<\epsilon<1/\zb$ that
\[d(F,\epsilon)\ =\ \limsup_{x\to\infty}\sup_{k<\epsilon x}\frac{\Prob(S_{k}>x)}{k\Fb(x)}\ < \ \infty\]
\end{corollary}
\begin{proof} Define $\delta=(1-\epsilon\zb)/\epsilon$. We can write $x=k\zb+y$ where
\[y\ =\ y(x,k)\ =\ x-k\zb\ \ge\ x(1-\epsilon \zb)\ =\ x\epsilon\delta\ \ge\ \delta k\,.\] 
Lemma~\ref{Lemma:11.7aa} therefore gives that for all large $x$ we can bound $\Prob(S_{k}>x)$
by $Ck\Fb(y)$ where $C$ does not depend on $x$. Now just note that by RV
\[\Fb(y) \ \le\ \Fb(x\epsilon\delta)\ \sim (\epsilon\delta)^{-\alpha}\Fb(x)\,.\]
\end{proof}

\begin{proof}[Proof of Proposition~\ref{Prop:2.8b}]
For ease of exposition, we start by the case $p=2$. We split the probability in question into four parts
\begin{align*}p_1(x)\ &=\ \Prob\bigl(S_{N_1}^{(1)}+S_{N_2}^{(2)}>x,\, N_1\le \epsilon x,\, N_2\le \epsilon x\bigr)\\
p_{21}(x)\ &=\ \Prob\bigl(S_{N_1}^{(1)}+S_{N_2}^{(2)}>x,\, N_1> \epsilon x,\, N_2\le \epsilon x\bigr)\\
p_{22}(x)\ &=\ \Prob\bigl(S_{N_1}^{(1)}+S_{N_2}^{(2)}>x,\, N_1\le \epsilon x,\, N_2> \epsilon x\bigr)\\
p_3(x)\ &=\ \Prob\bigl(S_{N_1}^{(1)}+S_{N_2}^{(2)}>x,\, N_1> \epsilon x,\, N_2> \epsilon x\bigr)
\\ \intertext{Here}
p_1(x)\ &=\ \sum_{k_1,k_2=0}^{\epsilon x}\Prob\bigl(S_{k_1}^{(1)}+S_{k_2}^{(2)}>x\bigr)
\Prob(N_1=k_1,N_2=k_2)
\end{align*}
Since $S_{k_1}^{(1)},\,S_{k_2}^{(2)}$ are independent, we have by standard RV theory that
\[\Prob\bigl(S_{k_1}^{(1)}+S_{k_2}^{(2)}>x\bigr)\ \sim\ (k_1c_1+k_2c_2)\Fb(x)\]
as $x\to\infty$. Further Corollary~\ref{Cor:2.8a} gives that for $k_1,k_2\le \epsilon x$ and all large $x$ we have
\begin{align*}\Prob\bigl(S_{k_1}^{(1)}+S_{k_2}^{(2)}>x\bigr)\ &\le\ 
\Prob\bigl(S_{k_1}^{(1)}>x/2\bigr)+\Prob\bigl(S_{k_2}^{(2)}>x/2\bigr)\\ &\le\ 
2\bigl(d(F_1,2\epsilon)k_1+d(F_2,2\epsilon)k_2\bigr)\Fb(x)\,.
\end{align*}
Hence by dominated convergence
\[\frac{p_1(x)}{\Fb(x)}\ \to \ \sum_{k_1,k_2=0}^{\infty}(k_1c_1+k_2c_2)
\Prob(N_1=k_1,N_2=k_2)\ =\ c_1\Exp N_1+c_2\Exp N_2\,.\]

For $p_3(x)$, denote by $A_j(m)$ the event that $S_{k_j}^{(j)}/k_j\le\zb_j/(1-\epsilon)$
for all $k_j>m$. Then  by the LNN there are  constants $r(m)$ converging to 0 as $m\to\infty$ 
such that $\Prob \rev{\bigl(}A_j(m)^\compl\rev{\bigr)} \le r(m)$ for $j=1,2$. It follows that
\begin{align*}
p_3(x)\ &\le \ \bigl(\Prob \rev{\bigl(}A_1(\epsilon x)^\compl\rev{\bigl)}+
\Prob\rev{\bigl(} A_2(\epsilon x)^\compl\rev{\bigl)}\bigr)
\Prob(N_1> \epsilon x,\, N_2> \epsilon x)\\ &
+\Prob\bigl(S_{N_1}^{(1)}+S_{N_2}^{(2)}>x,\, N_1> \epsilon x,\, N_2> \epsilon x,A_1(\epsilon x),A_2(\epsilon x)\bigr)
\\ &\le r(\epsilon x)\Oh\bigl(\Fb(x)\bigr)+\Prob\bigl((\zb_1N_1+\zb_2N_2)/(1-\epsilon)>x,\, N_1> \epsilon x,\, N_2> \epsilon x)\\ &\le\ \oh\bigl(\Fb(x)\bigr) \Prob\bigl((\zb_1N_1+\zb_2N_2)>\eta x,\, N_1> \epsilon x,\, N_2> \epsilon x)
\end{align*}
\rev{as $x\to\infty$},
where $\eta<1-\epsilon$ will be specified later.

For $p_{21}(x)$, we write $p_{21}(x)=p_{21}'(x)+p_{21}''(x)$ where
\begin{align*}
p_{21}'(x)\ &=\ \Prob\bigl(S_{N_1}^{(1)}+S_{N_2}^{(2)}>x,S_{N_2}^{(2)}\le \gamma x,\, N_1> \epsilon x,\, N_2\le \epsilon x\bigr)\\
p_{21}''(x)\ &=\ \Prob\bigl(S_{N_1}^{(1)}+S_{N_2}^{(2)}>x,S_{N_2}^{(2)}> \gamma x,\, N_1> \epsilon x,\, N_2\le \epsilon x\bigr)
\end{align*}
with $\gamma=2\epsilon\zb_2$. Here
\begin{align*}
p_{21}''(x)\ &\le \ \Prob\bigl(S_{N_1}^{(1)}+S_{\epsilon x}^{(2)}>x,S_{\epsilon x}^{(2)}> \gamma x,\, 
N_1> \epsilon x,\, N_2\le \epsilon x\bigr)\\ &\le\ 
\Prob\bigl(S_{ \epsilon x}^{(2)}> \gamma x,\, 
N_1> \epsilon x\bigr)\ =\ \Prob\bigl(S_{ \epsilon x}^{(2)}> \gamma x\Bigr)\, 
\Prob(N_1> \epsilon x)\\ &=\ \oh(1)\Oh\bigl(\Fb(x)\bigr)\ =\ \oh\bigl(\Fb(x)\bigr)\,,
\end{align*}
using the LLN in the fourth step. Further as in the estimates above
\begin{align*}
p_{21}'(x)\ &\le\ \Prob\bigl(S_{N_1}^{(1)}>x(1-\gamma),\, N_1> \epsilon x,\, N_2\le \epsilon x\bigr)\\ &\le \ 
\oh\bigl(\Fb(x)\bigr)\ +\  \Prob\bigl(\zb_1N_1>x(1-\gamma)(1-\epsilon),\, N_1> \epsilon x,\, N_2\le \epsilon x\bigr)
\\ &\le \  \Prob\bigl(\zb_1N_1+\zb_2N_2>x(1-\gamma)(1-\epsilon),\, N_1> \epsilon x,\, N_2\le \epsilon x\bigr)
\end{align*}

We can now finally put the above estimates together. For ease of notation, write $\eta=\eta(\epsilon)=
(1-\gamma)(1-\epsilon)$ and note that $\eta\uparrow 1$ as $\epsilon\downarrow 0$. Using a similar
estimate for  $p_{12}(x)$ as for  $p_{21}(x)$ and noting that
\[\Prob\bigl(\zb_1N_1+\zb_2N_2>\eta x,\, N_1\le \epsilon x,\, N_2\le \epsilon x\bigr)\ =\ 0\]
for $\epsilon$ small enough, we get
\begin{align*} \MoveEqLeft \limsup_{x\to\infty}\frac{1}{\Fb(x)}\Prob\bigl(S_{N_1}^{(1)}+S_{N_2}^{(2)}>x\bigr)\\ &
=\ c_1\Exp N_1+c_2\Exp N_2+ \limsup_{x\to\infty}\frac{1}{\Fb(x)}\Prob\bigl(\zb_1N_1+\zb_2N_2>\eta x\bigr)\\
&=\ c_1\Exp N_1+c_2\Exp N_2 +c(\zb_1,\zb_2) \limsup_{x\to\infty}\frac{\Fb(\eta x)}{\Fb(x)}\\
&=\ c_1\Exp N_1+c_2\Exp N_2 +c(\zb_1,\zb_2)\frac{1}{\eta^\alpha}
\end{align*}
Letting $\epsilon\downarrow 0$ gives that the $\limsup$ is bounded by $c_0\rev{+}c(\zb_1,\zb_2)$. Similar estimates
for the $\liminf$ complete the proof for $p=2$.

If $p>2$, the only essential difference is that $p_{21}(x),p_{22}(x)$ need to be replaced by the $2^p-2$ terms
corresponding to all combinations of some $N_i$ being $\le \epsilon x$ and the others $>\epsilon x$,
with the two exceptions being the ones where either all are $\le \epsilon x$ or all are $>\epsilon x$.
However, to each of these similar estimates as the above ones for $p_{21}(x)$ apply.
\end{proof}

\section{Preservation of MRV under sum operations}\label{S:PresMRV}
\setcounter{equation}{0}

Before giving our main auxiliary result, Proposition~\ref{Prop:5.8c}, it is instructive to recall two extremely simple example of MRV. The first is two i.i.d.\ RV$(F)$ r.v.'s $X_1,X_2$, where a big value of the $X_1+X_2$ 
can only occur if one variable is big and the other small, which gives MRV with the angular measure
concentrated on the points $(1,0),\,(0,1)\in \BB_2$ with mass 1/2 for each. Slightly more complicated:
\begin{proposition}\label{Prop:5.8b}
Let $N,Z,Z_1,Z_2,\ldots$ be non-negative r.v.'s such that $N\in\nat$, $Z,Z_1,Z_2,\ldots$ are i.i.d., non-negative
and independent of $N$. Assume that $\Prob(N>x)\sim c_N\Fb(x)$, $\Prob(Z>x)\sim c_Z\Fb(x)$ 
for some RV \rev{tail $\Fb(x)=L(x)/x^\alpha$} and write $S=\sum_1^NZ_i$, $\nb=\Exp N$, $\zb=\Exp Z$, 
where $c_N+c_Z>0$. Then:\\[1mm]
\emph{(i)} $\Prob(S>x)\sim (c_N \zb^{\alpha}+c_Z\nb)\Fb(x)$;\\[1mm]
\emph{(ii)} The random vector $(N,S)$ is \emph{MRV} with \[\Prob\bigl(\|(N,S)\|>x\bigr)\sim c_{N,S}\Fb(x)\quad
\text{where}\quad c_{N,S}=c_N (1+\zb^{\alpha})+c_Z\nb\]
and angular measure $\mu_{N,S}$ concentrated on the points $\bfb_1=\bigl(1/(1+\zb),\zb/(1+\zb)\bigr)$ and $\bfb_2=(0,1)$
with \[\mu_{N,S}(\bfb_1)=\frac{c_N}{c_N+c_Z\nb}\,,\qquad \mu_{N,S}(\bfb_2)=\frac{c_Z\nb}{c_N+c_Z\nb}\,.\]
\end{proposition}
\begin{proof} Part (i) is Lemma 4.7 of~\cite{Fay} (see also~\cite{Denisov}). The proof in~\cite{Fay} also shows that is $S>x$, then
either approximately $N\zb>x$, occuring w.p.\ $c_N\Fb(x/\zb)\sim c_N \zb^{\alpha}\Fb(x)$, or $N\le \epsilon x$ and  $Z_i>x$,
 occuring w.p.\ $c_Z\Exp[N\wedge \epsilon x]\Fb(x)$. The first possibility is what gives the atom of $\mu_{N,S}$ at $b_1$
 and the second  gives the atom  at $b_2$ since $\Exp[N\wedge \epsilon x]\uparrow\nb$.
 \end{proof}
 
\begin{proposition}\label{Prop:5.8c}
Let $\bfV=(\bfT,N)\in[0,\infty)^{p}\times\nat$ satisfy \MRVF,
let $\bfZ,\bfZ_1,\bfZ_2,\ldots$\,$\in[0,\infty)^q$ be i.i.d.\ and independent of $(\bfT,N)$ and 
satisfying \MRVF, and
define $\bfS=\sum_1^N\bfZ_i$. Then $\bfV^*=(\bfT,N,\bfS)$ satisfies \MRVF.
\end{proposition}  
\begin{proof} 
Let $\overline \bfz\in[0,\infty)^q$ be the mean of $\bfZ$.
Similar arguments as in Section~\ref{S:FayeMRV} show that $\|V^*\|>x$ basically occurs when either 
$\|\bfT\|+N+N\|\overline\bfz\|>x$ or when $ \|\bfV\|\le\epsilon x$ and some $\|\bfZ_i\|>x$.
The probabilities of these events are approximately of the form $c'\Fb(x)$ and $c''\Fb(x)$, so the radial part
of $\bfV^*$ is RV with asymptotic tail $c_{\bffV^*}\Fb(x)$ where $c_{\bffV^*}=c'+c''$.
Now
\[\Prob\Bigl(\frac{(\bfT,N)}{\|(\bfT,N)\|}\in \cdot\,\Big|\, \|\bfT\|+N+N\|\overline\bfz\|>x\Big)\ \to\ \mu'\]
for some probability measure $\mu'$ on the $(p+1)$-dimensional  unit sphere $\BB_{p+1}$; this follows since
$\|\bfT\|+N+N\|\overline\bfz\|$ is a norm and the MRV property of a vector is independent of the choice of norm.
Letting $\delta'_0$ be Dirac measure at $(0,\ldots,0)\in\RL^q$, $\delta''_0$ be Dirac measure at $(0,\ldots,0)\in\RL^{p+1} $
and $\mu''=\mu_{\bfZ}$ \rev{the angular measure of $\bfZ$}, we obtain the desired   conclusion with $c_{\bffV^*}=c'+c''$ and \rev{the angular measure of $\bfV^*$
given by}
\[\mu_{\bffV^*}\ =\ \frac{c'}{c'+c''}\,\mu'\otimes\delta''_0+\frac{c''}{c'+c''}\,\delta'_0\otimes\mu''\]
\end{proof}

In calculations to follow (Lemma~\ref{Lemma:7.8a}), extending some $\bfV$ to some $\bfV^*$ in a number of steps,  expressions for $c_{\bffV^*},\mu_{\bffV^*}$  can be deduced along the lines of the proof of
Propositions~\ref{Prop:5.8b}--\ref{Prop:5.8c} but the expression and details become extremely tedious. Fortunately, they won't be needed and are therefore omitted --- all that matters is existence.
If $\alpha$ is not an even integer,
the MRV alone of $\bfV^*$ can alternatively (and slightly easier) be obtained  from
Theorem 1.1(iv) of~\cite{Basrak}, stating that by non-negativity it suffices to verify MRV of any linear combination.

\section{Proof of Theorem~\ref{Th:6.8a} completed}\label{S:ProofCompl}
\setcounter{equation}{0}

\begin{lemma}\label{Lemma:7.8a} In the setting of \eqref{30.6aa},
the random vector \[\bfV^*(i)\ =\ \bigl(\widetilde Q(i), \widetilde N^{(1)}(i),\ldots,\widetilde N^{(K-1)}(i)\bigr)\] satisfies \MRVF\ 
for all $i$.
\end{lemma}
\begin{proof}
Let $\bigl|\tree_m(i)\bigr|$ be the
number of elements of $\tree_m(i)$ and 
\begin{align*} M_1(i)\ &=\ \sum_{m=1}^{N^{(K)}(i)}\bigl|\tree_m(i)\bigr|\,,\\
M_2(i)\ &=\ \sum_{m=1}^{N^{(K)}(i)}\sum_{g\in\tree_m(i)}\bigl(Q_g(K),N_g^{(1)}(K),\ldots,N_g^{(1)}(K-1)\bigr)\end{align*}
Recall that our basic assumption is that the
\begin{equation}\label{7.8d} \bfV^*(i)\ =\ \bigl(Q(i), N^{(1)}(i),\ldots,N^{(K)}(i)\bigr)\end{equation} 
satisfy \MRVF. The connection to a Galton-Watson tree
and Theorem~\ref{Th:3.8a} with $Q\equiv 1$, $N=N^{(K)}(i)$ therefore imply that so does any $\bigl|\tree_m(i)\bigr|$,
and since these r.v.'s are i.i.d.\ and independent of $N^{(K)}(i)$, Proposition~\ref{Prop:5.8c} gives that
$\bfV_1(i)=\bigl(\bfV(i), M_1(i)\bigr)$ satisfies \MRVF. Now the \MRVF\ property of \eqref{7.8d} with $i=K$
implies that the vectors $\bigl(Q_g(K),N_g^{(1)}(K),\ldots,N_g^{(K-1)}(K)\bigr)$, being distributed as 
$\bigl(Q(K),N^{(1)}(K),\ldots,N^{(K-1)}(K)\bigr)$
again satisfy \MRVF. But $M_2(i)$ is a sum of $M_1(i)$ such vectors that are i.i.d.\ given $M_1(i)$.
Using Proposition~\ref{Prop:5.8c} once more gives that $\bfV_2(i)=\bigl(\bfV(i), M_1(i),M_2(i)\bigr)$ satisfies \MRVF.
But $\bfV^*(i)$ is a function of $\bfV_2(i)$. Since this function is linear, property \MRVF\ of
$\bfV_2(i)$ carries over to $\bfV^*(i)$.
\end{proof}

\begin{proof}[Proof of Theorem~\ref{Th:6.8a}]
We use induction in $K$. The case $K=1$ is just Theorem~\ref{Th:3.8a}, so assume 
Theorem~\ref{Th:6.8a} shown for $K-1$.

The induction hypothesis and Lemma~\ref{Lemma:7.8a}  implies that $\Prob(R(i)>x)\sim d_i\Fb(x)$ for $i=1,...,K-1$.
Rewriting \eqref{30.6a} for $i=K$ as
\begin{equation*} R(K) \ \eqdistr\  Q^*(K)+\sum_{m=1}^{N^{(K)}(K)}R_{m}(K)
\ \ \text{where}\ \  Q^*(K)\,=\,\sum_{k=1}^{K-1}\sum_{m=1}^{N^{(k)}(K)}R_{m}(k)\,,\end{equation*}
we have a \rev{fixed-point} problem of type \eqref{eq1} and can then use Theorem~\ref{Th:6.8a} to conclude that
also $\Prob(R(K)>x)\sim d_K\Fb(x)$, noting that the needed MRV condition on $\bigl(Q^*(K),N^{(k)}(K)\bigr)$
follows by another application of Proposition~\ref{Prop:5.8c}.

Finally, to identify the $d_i$ via \eqref{6.8d}, appeal to Proposition~\ref{Prop:2.8b} with
$\bfN=\bigl(Q(i), N^{(1)}(i),\ldots,N^{(K)}(i)\bigr)$, writing the r.h.s.\ of \eqref{30.6a} as
\[\Oh(1)\,+\,\sum_{m=1}^{\lfloor Q(i)\rfloor} 1\,+\, \sum_{k=1}^K\sum_{m=1}^{N^{(k)}(i)}R_{m}(k)\,.\]
Existence and uniqueness of a solution to \eqref{6.8d} follows by once more noticing that $\rho<1$
implies that $\bfI-\bfM$ is invertible.
\end{proof}


\appendix

\section{\rev{Proof of \eqref{30.6ax}}}

\rev{The RV of linear combinations subject to MRV assumptions has received considerable attention, see e.g.~\cite{Basrak}, but we could not find explicit
formulas like \eqref{30.6ax} for the relevant constants so we give a
self-contained proof. The formula is a special case of the following: if $\bfX=(X_1\,\ldots\,X_n)\in\RL^n$
is a random vector such that $\Prob\bigl(\|\bfX\|>t\big)\sim L(t)/t^\alpha$
and $\bfTheta=\bfX/\|\bfX\|$ has   conditional limit distribution 
$\mu$ in ${\mathcal B}_1$ given $\|\bfX\|>t$ as $t\to\infty$, then
\[\Prob(\bfa\cdot\bfX>x)\ =\ \Prob(a_1X_1+\cdots+a_nX_n>t)\ \sim\ 
\frac{L(t)}{t^\alpha} \int_{{\mathcal B}_1}\indi(\bfa\cdot\bftheta>0)(\bfa\cdot\bftheta)^\alpha\mu(\dd\bftheta)\]}

\rev{To see this, note that given $\bfTheta=\bftheta\in{\mathcal B}_1$,
$\bfa\cdot\bfX =$ $ \|\bfX\|(\bfa\cdot\bftheta)$
will exceed $t>0$ precisely when $\bfa\cdot\bftheta>0$ and $\|\bfX\|>t/\bfa\cdot\bftheta$.
Thus one expects that
\begin{align*}\MoveEqLeft
\Prob(\bfa\cdot\bfX>t)\ \sim\ 
 \int_{{\mathcal B}_1}\indi(\bfa\cdot\bftheta>0)\Prob\bigl(\|\bfX\|>t/\bfa\cdot\bftheta\bigr)\,\mu(\dd\bftheta)
 \\ & \sim\ 
 \int_{{\mathcal B}_1}\indi(\bfa\cdot\bftheta>0)\frac{L(t/\bfa\cdot\bftheta)}{(t/\bfa\cdot\bftheta)^\alpha}\mu(\dd\bftheta)
 \  \sim\ 
\frac{L(t)}{t^\alpha} \int_{{\mathcal B}_1}\indi(\bfa\cdot\bftheta>0)(\bfa\cdot\bftheta)^\alpha\mu(\dd\bftheta)
 \end{align*}
 which is the same as asserted.}
 
 \rev{For the rigorous proof, assume $\|\bfa\|=1$. Then ${\mathcal B}_1$ is the disjoint union of the sets
 $B_{1,n},\ldots,B_{n,n}$ where $B_{i,n}=$ 
 $\{\bftheta\in{\mathcal B}_1:\,(i-1)/n<\bfa\cdot\bftheta\le i/n\}$
 for $i=2,\ldots,n$ and $B_{1,n}=$ $\{\bftheta\in{\mathcal B}_1:\,\bfa\cdot\bftheta\le 1/n\}$. Assuming $\Prob(\bfTheta=i/n)=0$ for all integers $i,n$, we get
 \begin{align*}
\Prob(\bfa\cdot\bfX>t)\ &=\ \sum_{i=1}^n \Prob\bigl(\bfa\cdot\bfX>t,\bfTheta\in B_{i,n}\bigr)\
\le\ \sum_{i=1}^n \Prob\bigl(\|\bfX\|>ti/n,\bfTheta\in B_{i,n}\bigr)
\\& \sim\ \sum_{i=1}^n \frac{L(ti/n)}{(ti/n)^\alpha} 
\Prob\bigl(\bfTheta\in B_{i,n}\,\big|\,|\bfX\|>t\bigr)\
\sim\ \frac{L(t)}{t^\alpha} \sum_{i=1}^n (i/n)^\alpha \Prob\bigl(\bfTheta\in B_{i,n}\bigr)\\  &=\ 
\frac{L(t)}{t^\alpha} \int_{{\mathcal B}_1}f_{+,n}(\bftheta)\,\mu(\dd\bftheta)
\end{align*}
where $ f_{+,n}$ is the step function taking value $(i/n)^\alpha$ on $B_{i,n}$.
A similar argument gives the asymptotic lower bound $\int \!f_{-,n}\,\dd\mu\,L(t)/t^\alpha$ for 
$\Prob(\bfa\cdot\bfX>t)$ where $ f_{-,n}$ equals $\bigl((i-1)/n)\bigr)^\alpha$ on $B_{i,n}$
for $i>1$ and $0$ on $B_{1,n}$.
But $f_{\pm,n}(\bftheta)$ both have limits $\bigl[(\bfa\cdot\bftheta)^+\bigr]^\alpha$ 
as $n\to\infty$ and are
bounded by 1. Letting $n\to\infty$ and using dominated convergence completes the proof.}

\rev{The case  $\Prob(\bfTheta=i/n)>0$ for some $i,n$ is handled by a trivial
redefinition of the $B_{i,n}$.}

\end{document}